\documentclass[letterpaper,12pt]{amsart}
\usepackage{amsmath,amssymb,amsxtra,amsthm, amstext, amscd,amsfonts,fancyhdr, hyperref,enumerate}
\usepackage[OT2,T1]{fontenc}
\DeclareSymbolFont{cyrletters}{OT2}{wncyr}{m}{n}
\DeclareMathSymbol{\Sha}{\mathalpha}{cyrletters}{"58}

\newcommand{\bQ}{{\mathbb{Q}}}
\newcommand{\bR}{{\mathbb{R}}}

\newcommand{\bZ}{{\mathbb{Z}}}

  \newcommand{\R}{{\mathcal{R}}}

\newcommand{\GL}{\operatorname{GL}}

\newcommand{\Aut}{\operatorname{Aut}}

\newcommand{\ep}{\varepsilon}

\newcommand{\upchi}{{\raise.35ex\hbox{$\chi$}}}

\makeatletter
\@namedef{subjclassname@2010}{%
  \textup{2010} Mathematics Subject Classification}
\makeatother

\newtheorem{theorem}{Theorem}[section]
\newtheorem{corollary}[theorem]{Corollary}

\newtheorem{lemma}[theorem]{Lemma}

\theoremstyle{definition}

\numberwithin{equation}{section}

\begin{document}

\title{On the representation of integers by binary quadratic forms}

\author{Stanley Yao Xiao}
\address{Mathematical Institute \\
University of Oxford \\
Oxford \\  OX2 6GG \\
United Kingdom}
\email{stanley.xiao@maths.ox.ac.uk}
\indent

%%%%%%%%%%%%%%%%%%%%%%%%%%%%%%%%%%%%%%%%%%%%%%%%%%%%%%%%%%%%%%%%%%%

\begin{abstract} In this note we show that for a given irreducible binary quadratic form $f(x,y)$ with integer coefficients, whenever we have $f(x,y) = f(u,v)$ for integers $x,y,u,v$, there exists a rational automorphism of $f$ which sends $(x,y)$ to $(u,v)$. 
\end{abstract}

\maketitle

%\newpage

\section{Introduction} 

Let $F$ be a binary form with integer coefficients, non-zero discriminant, and degree $d \geq 2$. We say that an integer $h$ is \emph{representable} by $F$ if there exist integers $x,y$ such that $F(x,y) = h$. It is an old question, dating back to Diophantus in the case of sums of two squares, to determine which integers $h$ are representable by a given form $F$. While an exact description (for example, in terms of congruence conditions) remain elusive for all but the simplest of cases, asymptotic results have now been established. Define
\begin{equation} \label{RFZ} \R_F(Z) = \{h \in \bZ : h \text{ is representable by } F, |h| \leq Z\} \end{equation}
and $R_F(Z) = \# \R_F(Z)$. Landau proved in 1908 that there exists a positive number $C_1$ such that
\begin{equation} \label{2S} R_{x^2 + y^2}(Z) \sim \frac{C_1 Z}{\sqrt{\log Z}},\end{equation}
and shortly after the result was established for all positive definite binary quadratic forms. \\ \\
In general, one expects that for a binary form $F$ with degree $d \geq 3$, integer coefficients, and non-zero discriminant, that there exists a positive number $C(F)$ such that the asymptotic relation
\begin{equation} \label{main eq} R_F(Z) \sim C(F) Z^{\frac{2}{d}} \end{equation} 
holds. It would take over half a century before the analogous asymptotic formula would be established for non-abelian cubic forms, which was achieved by Hooley. He proved in \cite{Hoo1} that (\ref{main eq}) holds whenever $F$ is a non-abelian binary cubic form. In subsequent works \cite{Hoo2} \cite{Hoo3}, he established (\ref{main eq}) for bi-quadratic binary quartic forms and abelian binary cubic forms, respectively. In \cite{SX}, Stewart and Xiao established (\ref{main eq}) for all integral binary forms of degree $d \geq 3$ and non-zero discriminant. \\ \\
For $F$ a binary form of degree $d \geq 2$, define
\[\Aut_\bQ F = \left\{T = \left(\begin{smallmatrix} t_1 & t_2 \\ t_3 & t_4 \end{smallmatrix} \right) \in \GL_2(\bQ) : F(x,y) = F(t_1 x + t_2 y, t_3 x + t_4 y)  \right \}.\]
The absence of the logarithmic term in (\ref{main eq}) as opposed to (\ref{2S}) is accounted for by the fact that for a binary form $F$ of degree at least $3$, $\Aut_\bQ F$ is always finite. When $F$ is a quadratic form, the group $\Aut_\bQ F$ is infinite. \\ \\
We say a representable integer $h$ is \emph{essentially represented} if whenever $(x,y), (u,v) \in \bZ^2$ are such that $F(x,y) = F(u,v) = h$, there exists $T \in \Aut_\bQ F$ such that $\binom{x}{y} = T \binom{u}{v}$. Note that if $F(x,y) = h$ has a unique solution, then $h$ is essentially represented since $\left(\begin{smallmatrix}1 & 0 \\ 0 & 1 \end{smallmatrix} \right) \in \Aut_\bQ F$. Put
\[\R_F^{(1)}(Z) = \{h \in \R_F(Z) : h \text{ is essentially represented}\}\]
and $R_F^{(1)}(Z) = \# \R_F^{(1)}(Z)$. In the $d \geq 3$ case Heath-Brown showed in \cite{HB1} that there exists $\eta_d > 0$, depending only on the degree $d$, such that for all $\ep > 0$
\begin{equation} \label{essential} R_F(Z) = R_F^{(1)}(Z)\left(1 + O_\ep \left(Z^{-\eta_d + \ep} \right) \right).\end{equation}
This is essentially reduces the question of enumerating $\R_F(Z)$ to that of $\R_F^{(1)}(Z)$, which is far simpler, and the key to our success in \cite{SX}. Heath-Brown's theorem does not address the case of quadratic forms, which we do so now: 

\begin{theorem} \label{essential rep thm} Let $f$ be an irreducible and primitive binary quadratic form. Then every integer $h$ representable by $f$ is essentially represented. 
\end{theorem}

Consider the quadric surface $X_f$ defined by
\[X_f : f(x_1, x_2) = f(x_3, x_4).\]

In \cite{HB1}, Heath-Brown showed that lines on $X_f$ correspond to automorphisms of $f$, possibly defined over a larger field. His result and our Theorem \ref{essential rep thm} has the following consequence:

\begin{corollary} Let $X_f$ be the surfaced defined by $f(x_1, x_2) = f(x_3, x_4)$, with $f$ a binary quadratic form with integer coefficients and non-zero discriminant. Then every point in $X_f(\bQ)$ lies on a rational line contained in $X_f$. 
\end{corollary}

It has been pointed out to the author that Theorem \ref{essential rep thm} essentially follows from Witt's theorem (see Theorem 42.16 in \cite{JWS}). Nevertheless, we feel that this result is of independent interest to number theorists and does not appear to be well-known. 

\section{Preliminary lemmas} 

The strategy is very simple: for a given pair of integers $(x,y), (u,v)$ such that $f(x,y) = f(u,v)$, we exhibit an explicit automorphism of $f$ which sends $(x,y)$ to $(u,v)$. In fact, we will draw such an automorphism from a proper subgroup of $\Aut_\bQ f$. Put
\[f(x,y) = f_2 x^2 + f_1 xy + f_0 y^2,\]
and put
\[\delta = \left \lvert \frac{f_1^2 - 4 f_2 f_0}{4} \right \rvert.\]
We shall first characterize the automorphism group $\Aut_\bQ f$. It turns out that this depends on whether $f$ is positive definite or not. 

\subsection{Positive definite binary quadratic forms} In this case, we shall pick our $T$ from the group $\Aut_\bQ f \cap SO_f(\bR)$, where
\[SO_f(\bR) = \{T \in \GL_2(\bR) : \det T = 1, f_T = f\}.\]
The group $SO_f(\bR)$ is conjugate to the special orthogonal group $SO_2(\bR)$ and its elements look like 

\[T_f(t) = \begin{pmatrix} \cos t + \dfrac{f_1 \sin t}{2\sqrt{\delta}} & \dfrac{f_0 \sin t}{\sqrt{\delta}} \\ \\ \dfrac{-f_2 \sin t}{\sqrt{\delta}} & \cos t - \dfrac{f_1 \sin t}{2\sqrt{\delta}} \end{pmatrix}, t \in [0, 2\pi).\]
If we demand that $T_f(t) \in \GL_2(\bQ)$, then it follows that $\cos t \in \bQ$ and $\sqrt{\delta} \sin t \in \bQ$. Put
\[u = \cos t, v = \frac{\sin t}{\sqrt{\delta}}.\]
Then $u,v$ satisfy the equation
\[u^2 + \delta v^2 = 1.\]
Put $E_\delta$ for the curve defined by
\begin{equation} \label{ellipse} E_\delta : x^2 + \delta y^2 = 1.\end{equation}
We then see that there is a bijection between rational points on $E_\delta$ and rational elements $T \in \Aut_\bQ f$. We now characterize the set of rational points on $E_\delta$.

\begin{lemma} \label{ellipse lemma} Let $E_\delta$ be the curve given by (\ref{ellipse}), with $4\delta$ a positive integer. Then the set of rational points on $E_\delta$ is given by the parametrization
\[\left(\frac{\delta p^2 - q^2}{\delta p^2 + q^2}, \frac{2pq}{\delta p^2 + q^2}\right), p, q \in \bZ, q > 0, \gcd(p,q) = 1.\] 
\end{lemma}

\begin{proof}Using the fact that $(1,0)$ is a point on the curve $E_\delta$, we use the slope method to find all other rational points. Indeed, the intersection of the line given by
\[y = m(x-1), m \in \bQ\]
and the curve $E_\delta$ is another rational point on $E_\delta$, and all such points arise this way. Substituting, we find that
\[x^2 + \delta (m(x-1))^2 = 1\]
is equivalent to
\[x = \frac{\delta m^2 \pm 1}{\delta m^2 + 1}.\]
The $+$ sign gives $x = 1$, and the $-$ sign gives
\[x = \frac{\delta m^2 - 1}{\delta m^2 + 1}\]
which corresponds to the point
\[(x,y) = \left(\frac{\delta m^2 - 1}{\delta m^2 + 1}, \frac{2m}{\delta m^2 + 1} \right).\]
If we write the slope $m$ as $m = p/q$, where $q > 0$ and $\gcd(p,q) = 1$, then the point can be given as
\[(x,y) = \left(\frac{\delta p^2 - q^2}{\delta p^2 + q^2}, \frac{2pq}{\delta p^2 + q^2} \right),\]
as desired. \end{proof} 

\subsection{Indefinite binary quadratic forms} In this case, the group $SO_f(\bR)$ is no longer connected, and we shall focus on the \emph{principal branch} of $SO_f(\bR)$, which is the branch containing the identity matrix. This branch can be identified as the set of matrices of the form
\[T_f(t) = \begin{pmatrix} \cosh t - \dfrac{f_1 \sinh t}{2 \sqrt{\delta}} & & -\dfrac{f_0 \sinh t}{\sqrt{\delta}} \\ \\ \dfrac{f_2 \sinh t}{\sqrt{\delta}} & & \cosh t + \dfrac{f_1 \sinh t}{2 \sqrt{\delta}} \end{pmatrix}, t \in \bR.\]
Again, if we demand that $T_f(t) \in \GL_2(\bQ)$, then necessarily $\cosh t , \sqrt{\delta} \sinh t \in \bQ$. Put
\[u = \cosh t, v = \frac{\sinh t}{\sqrt{\delta}}.\]
Notice that $(u,v)$ lies on the curve
\begin{equation} \label{hyperbola} E_\delta: x^2 - \delta y^2 = 1. \end{equation}
It is immediate that there is a bijection between the set of rational points $E_\delta(\bQ)$ and elements in $SO_f(\bQ)$. We have the following characterization of the rational points on $E_\delta$:

\begin{lemma} \label{hyperbola lemma} Let $E_\delta$ be the curve given by (\ref{hyperbola}). Then the set of rational points $E_\delta(\bQ)$ are given by the parametrization
\[\left(\frac{\delta p^2 + q^2}{\delta p^2 - q^2}, \frac{2pq}{\delta p^2 - q^2} \right), p,q \in \bZ, q > 0, \gcd(p,q) = 1.\]
\end{lemma}

\begin{proof} Same as Lemma \ref{ellipse lemma}. 
\end{proof}

\section{Proof of Theorem \ref{essential rep thm}}

We first address the case when $f$ is positive definite. Let $h$ be a representable integer of $f$. If there exists exactly one pair of integers $(x,y)$ such that $f(x,y) = h$, then $h$ is essentially represented. Now suppose there exist distinct representations $(x,y), (u,v)$ of $h$, so that
\begin{equation} \label{rep h} h = f(x,y) = f(u,v).\end{equation}
Put
\[m = 2 f_2 ux + f_1(uy + vx) + 2 f_0 vy - 2h, n = 2\delta(uy - vx)\]
and
\[T_f(m,n) = \frac{1}{\delta m^2 + n^2} \begin{pmatrix} \delta m^2 - n^2 + f_1 mn & 2 f_0 mn \\ -2 f_2 mn & \delta m^2 - n^2 - f_1 mn  \end{pmatrix} \in \Aut_\bQ f.\]
Observe that 
\[(\delta m^2 - n^2 + f_1 mn) x + 2f_0 mn y = hm \delta u\]
and
\[-2f_2 mn x + (\delta m^2 - n^2 - f_1 mn)y = hm \delta v.\]
Moreover, by expanding, we see that
\[\delta m^2 + n^2 = hm \delta.\]
It then follows that
\[T_f(m,n) \binom{x}{y} = \binom{u}{v}.\]

The proof of the theorem when $f$ is indefinite is similar, but we include the full argument for the sake of completeness. Suppose that (\ref{rep h}) holds and put
\[m = 2f_2 ux + f_1(uy + vx) + 2 f_0 vy - 2h, n = 2 \delta(vx - uy).\]
Then the associated $T_f(m,n) \in \Aut_\bQ f$ is given by
\[T_f(m,n) = \frac{1}{\delta m^2 - n^2} \begin{pmatrix} \delta m^2 + n^2 - f_1 mn &  -2 f_0 mn \\ 2 f_2 mn & \delta m^2 + n^2 + f_1 mn \end{pmatrix}.\]
A routine calculation then yields that
\[T_f(m,n) \binom{x}{y} = \binom{u}{v},\]
as desired. 

%\[ u = \frac{(m^2 - n^2)x + 2mn y}{m^2 + n^2}, v = \frac{-2mnx + (m^2 - n^2)y}{m^2 + n^2}.\]
%\[\frac{u}{v} = \frac{(m^2 - n^2)x + 2mny}{-2mnx + (m^2 - n^2)y}\]
%\[(m^2 - n^2)vx + 2mnvy = -2mnux + (m^2 - n^2)uy\]
%$s = m/n$
%\[s^2(vx - uy) + 2s(vy + ux) - (uy - vx) = 0.\]
%\[4(vy + ux)^2 + 4(uy - vx)^2 = 4(v^2 y^2 + 2uxvy + u^2 x^2 + u^2 y^2 - 2uvxy + v^2 x^2) \]
%\[= 4(x^2 + y^2)(u^2 + v^2) = 4 h^2 \]
%\[s = \frac{-(vy + ux) \pm h}{vx - uy}\]

\end{document}